\documentclass{amsart}

\usepackage{amsmath,amssymb,amsthm}
\usepackage{hyperref}
\usepackage{xcolor}
\usepackage{tikz-cd}

\usepackage[nohug,heads=LaTeX]{diagrams}

\newtheorem{prop}{Proposition}[section]
\newtheorem{thm}[prop]{Theorem}
\newtheorem{lem}[prop]{Lemma}

\theoremstyle{definition}

\newtheorem{rem}[prop]{Remark}

\newtheorem*{ack}{Acknowledgement}

\def\co{\colon\thinspace}

\newcommand{\C}{\mathbb C}

\newcommand{\rmd}{\mathrm d}

\newcommand{\F}{\mathbb F}

\newcommand{\MM}{\mathcal M}

\newcommand{\N}{\mathbb N}

\newcommand{\bfp}{\mathbf p}

\newcommand{\bfq}{\mathbf q}

\newcommand{\bfu}{\mathbf u}

\newcommand{\Z}{\mathbb Z}

\newcommand{\bfz}{\mathbf z}

\newcommand{\lra}{\longrightarrow}
\newcommand{\ra}{\rightarrow}

\DeclareMathOperator{\cp}{\mathrm{Cap}}

\DeclareMathOperator{\ev}{\mathrm{ev}}

\DeclareMathOperator{\fs}{\mathrm{FS}}

\DeclareMathOperator{\id}{\mathrm{id}}

\DeclareMathOperator{\Int}{\mathrm{Int}}

\DeclareMathOperator{\pt}{\mathrm{pt}}

\DeclareMathOperator{\st}{\mathrm{st}}

\begin{document}

\author{Myeonggi Kwon}
\author{Kai Zehmisch}
\address{Mathematisches Institut, Justus-Liebig-Universit\"at Gie{\ss}en,
Arndtstra{\ss}e 2, D-35392 Gie{\ss}en, Germany}
\email{Myeonggi.Kwon@math.uni-giessen.de, Kai.Zehmisch@math.uni-giessen.de}

\title[Fillings and fittings]{Fillings and fittings of unit cotangent bundles of odd-dimensional spheres}

\date{02.04.2019}

\begin{abstract}
  We introduce the concept of fittings to symplectic fillings
  of the unit cotangent bundle of odd-dimensional spheres.
  Assuming symplectic asphericity
  we show that all fittings
  are diffeomorphic to the respective unit co-disc bundle.
\end{abstract}

\thanks{This research is part of a project in the SFB/TRR 191
{\it Symplectic Structures in Geometry, Algebra and Dynamics}, 
funded by the DFG}

\maketitle

%%%%%%%%%%%%%%%%%%%%%%%%%%%%%%%%%%%%%%%%

\section{Introduction\label{sec:intro}}

%%%%%%%%%%%%%%%%%%%%%%%%%%%%%%%%%%%%%%%%
%%%%%%%%%%%%%%%%%%%%%%%%%%%%%%%%%%%%%%%%

There is a general question in symplectic geometry:
How much does the boundary of a compact symplectic manifold
know about its interior?
When the boundary is of contact type,
the question becomes meaningful.
In fact,
this question is of high current interest in dimension $4$.
But much less is known about the topology of the fillings in higher dimensions.
By results of
Eliashberg--Floer--McDuff, see \cite[Theorem 1.5]{mcd91},
and Barth--Geiges--Zehmisch \cite{bgz}
the diffeomorphism type of
symplectically aspherical fillings
is uniquely determined for
many subcritically Stein fillable contact manifolds.
But almost nothing is known in the {\it critical} case so far.
This paper is about symplectically aspherical fillings of
the unit cotangent bundle of odd-dimensional spheres.
These are never subcritically Stein by \cite[Theorem 1.2]{bgz}
because the unit cotangent bundle is of critical Stein type.
For basic notions the reader is referred to Geiges' book \cite{gei08}.

For a given natural number $d$
we consider the cotangent bundle $T^*S^{2d+1}$
of the round $(2d+1)$--dimensional sphere.
Like every cotangent bundle,
$T^*S^{2d+1}$ admits the Liouville canonical $1$-form $\lambda$,
which in local coordinates is given by $\bfp\,\rmd\bfq$.
The induced contact structure
on the unit cotangent bundle
\[
M=ST^*S^{2d+1}
\]
is denoted by $\xi$.
The resulting contact manifold $(M,\xi)$
admits a natural symplectic filling
$(W_{\!\st},\omega_{\st})$
given by the unit disc bundle
$\big(DT^*S^{2d+1},\rmd\lambda\big)$.
We call $(W_{\!\st},\omega_{\st})$
the {\bf standard filling}.

We would like to find natural conditions
under which a symplectic filling $(W,\omega)$ of $(M,\xi)$
is diffeomorphic to $W_{\!\st}$.
Observe that a blow up of $(W,\omega)$ at any interior point
results in a new symplectic filling of $(M,\xi)$,
which cannot be diffeomorphic to $W_{\!\st}$.
The reason is that diffeomorphically
the blow up is a connected sum with $\C P^{4d+2}$
with the opposite orientation
and therefore contains a homologically non-trivial
symplectic $2$-sphere.
In order to exclude this ambiguity
we require all fillings to be {\bf symplectically aspherical},
that is we require the symplectic form to be zero
evaluated on all spherical homology $2$-classes.

Unfortunately,
symplectic asphericity seems not to be rigid enough
to determine the diffeomorphism type
of symplectic fillings
of the unit cotangent bundle $(M,\xi)$ of $S^{2d+1}$.
That's why we want to search for an additional condition.
In dimension $4$ holomorphic curves as such are often self-dual
in a symplectic way and in terms of Poincar\'e duality.
This is crucial for uniqueness of symplectic fillings in dimension $4$.
Therefore,
we are looking for symplectic hypersurfaces
inside a given filling $(W,\omega)$
that behave nicely w.r.t.\ the boundary of $(W,\omega)$
and that remember a sufficiently large amount of symplectic informations
of $(W_{\!\st},\omega_{\st})$.

In the given situation this means the following:
The standard filling $(W_{\!\st},\omega_{\st})$ itself
appears as a Weinstein neighbourhood
of the graph of the anti-Hopf map inside $\C^{d+1}\times\C P^d$,
see Section \ref{sec:antihopf},
and embeds into the obvious partial compactification
$\C P^1\times\C^d\times\C P^d$.
The intersection of $W_{\!\st}$ with $\C P^1\times\C^d\times\C P^{d-1}$
defines a complex hypersurface in $(W_{\!\st},\omega_{\st})$
relative to the boundary.
This complex hypersurface
inherits the same homologically trivial intersection behaviour
as holomorphic spheres $\C P^1\times\pt\times\pt$
and complex hypersurfaces $\C P^1\times\C^d\times\C P^{d-1}$
have inside $\C P^1\times\C^d\times\C P^d$.
We call a filling that admits such a complex hypersurface
a {\bf fitting},
see Section \ref{sec:definition}.
As it will turn out fittings are the right concept to ensure uniqueness:

\begin{thm}
\label{thm:installabilitytheorem}
If $(W,\omega)$ is a symplectically aspherical fitting of $(M,\xi)$,
then $W$ and $W_{\!\st}$ are diffeomorphic.
\end{thm}

%%%%%%%%%%%%%%%%%%%%%%%%%%%%%%%%%%%%%%%%

\section{From homology groups to diffeomorphism type\label{sec:topologicalpart}}

In Section \ref{sec:holfillviacap} we will verify
the assumptions of the following proposition
for fittings.
The argumentation will be similar to earlier work
which can be found in \cite[Section 5]{bgz} and \cite[Section 4]{bschz19}.

\begin{prop}
\label{prop:homologyenough}
  Let $(W, \omega)$ be a symplectically aspherical filling of $(M,\xi)$.
  If $W$ is simply connected
  and has the homology type of $W_{\!\st}$,
  i.e.\ $H_kW=H_kW_{\!\st}$ for all $k$,
  then $W$ is diffeomorphic to $W_{\!\st}$.
  In other words,
  the diffeomorphism type of $W$ is unique in this case.
\end{prop}

This section is intended to provide evidence for this statement.

%%%%%%%%%%%%%%%%%%%%%%%%%%%%%%%%%%%%%%%%

\subsection{A cobordism\label{subsec:acobordism}}

We consider the cotangent disc bundle
$D_RT^*S^{2d+1}$
of radius $R>1$
so that the unit cotangent disc bundle
$W_{\!\st}=DT^*S^{2d+1}$
appears as a subset of $D_RT^*S^{2d+1}$.
Via gluing of smooth manifolds with boundary
we define 
\[
W_1:=
\big(D_RT^*S^{2d+1}\setminus\Int W_{\st}\big)
\cup_MW
\]
by replacing $W_{\!\st}$ with $W$ inside
$D_RT^*S^{2d+1}$.
For sufficiently large radius $R$
the spherical shell bundle
$\big(D_RT^*S^{2d+1}\setminus\Int W_{\st}\big)$
contains a copy $W_0$ of $W_{\!\st}$.
This copy of $W_{\!\st}$ is obtained by
first taking a section $s$
of the sphere bundle $ST^*S^{2d+1}\ra S^{2d+1}$
and then shifting $DT^*S^{2d+1}$ suitably
inside $T^*S^{2d+1}$ in the direction of $s$.
This shifted copy of $W_{\!\st}$
is also considered to be contained in $W_1$.
We define a cobordism $X$
between $M_0:=\partial W_0$ and $M_1:=\partial W_1$ by
\[
X:=W_1\setminus\Int W_0\,.
\]

\begin{rem}
 The construction of the cobordism $X$
 heavily relies on the existence
 of a nowhere vanishing vector field on $S^{2d+1}$.
 This is the first of two places in our argument
 that requires the spheres $S^{2d+1}$ to have odd dimensions.
 The second one is the use of the Hopf fibration in Section \ref{sec:antihopf}. 
\end{rem}

%%%%%%%%%%%%%%%%%%%%%%%%%%%%%%%%%%%%%%%%

\subsection{Applying the $h$-cobordism theorem\label{subsec:applhcobordism}}

Under the assumptions of Proposition \ref{prop:homologyenough}
the cobordism $X$ is an $h$-cobordism:
The boundary components $M_0$ and $M_1$
are both $S^{2d}$-bundles over $S^{2d+1}$, $d\geq1$, and,
hence, simply connected thanks to the corresponding homotopy sequence.
The cobordism $X$ itself is simply connected too.
Indeed, $W_1\simeq W$ is simply connected by assumption.
The contraction $S^{2d+1}\simeq W_{\st}$
induces a contraction $S\simeq W_0$
of $W_0$ onto the shifted copy $S$ of $S^{2d+1}$
under the shift that we used in the construction of $X$
in Section \ref{subsec:acobordism}.
This contraction defines a contraction of $W_1\setminus S$ onto $\Int X$.
Simply connectedness of $X$ now follows
with the simply connectedness of $W_1$
together with a general position argument that uses
$2+(2d+1)<4d+2=\dim W_1$.
It follows that the inclusions
$M_0\hookrightarrow X$ and $M_1\hookrightarrow X$
of the boundaries
induce isomorphisms of the respective fundamental groups.

\begin{lem}
\label{lem:hcobor}
$H_*(X,M_0)=0$.
\end{lem}

\begin{proof}[{\bf Proof of Proposition \ref{prop:homologyenough}.}]
With Lemma \ref{lem:hcobor},
Poincar\'e--Lefschetz duality and the universal coefficient theorem
imply $H_*(X, M_1)=0$.
A combination of Hurewicz' and Whitehead's theorem implies
that the inclusions $M_0 \hookrightarrow X$
and $M_1 \hookrightarrow X$ are strong deformation retracts.
This uses our preceding considerations
on the simply connectedness of $M_0$, $M_1$, and $X$.
In other words $\{M_0,X,M_1\}$ is a simply connected $h$-cobordism
of dimension $4d+2\geq6$.

The $h$-cobordism theorem implies that
$X$ is diffeomorphic to $[0,1]\times M_0$
via a diffeomorphism that restricts to the identity
on $0\times M_0$, see \cite{miln65}.
Gluing back $W_0$ to $X$ and extending the diffeomorphism suitably
we obtain that $W_1$ is diffeomorphic to a collar extension of $W_0$.
It follows that $W$ and $W_{\!\st}$ are diffeomorphic. 
\end{proof}

%%%%%%%%%%%%%%%%%%%%%%%%%%%%%%%%%%%%%%%%

\subsection{Killing relative homology}

It remains to prove Lemma \ref{lem:hcobor}.
By the excision axiom this lemma is equivalent
to $H_*(W_1,W_0)=0$.

To show homology vanishing of the pair $(W_1,W_0)$
we observe that $S^{2d+1}\simeq W_0$,
so that, invoking our assumptions,
$H_kS^{2d+1}=H_kW_0=H_kW_1$ for all $k\in\Z$.
The long exact sequence of $(W_1,W_0)$
implies $H_k(W_1,W_0)=0$ for all $k\neq2d+1,2d+2$
and exactness of
\[
0\ra
H_{2d+2}(W_1,W_0)\ra
H_{2d+1}W_0\ra
H_{2d+1}W_1\ra
H_{2d+1}(W_1,W_0)\ra
0\,.
\]
Therefore, in order to prove 
Proposition \ref{prop:homologyenough}
it suffices to show that the map
$\iota_*\co H_{2d+1}W_0\ra H_{2d+1}W_1$
induced by the inclusion $W_0\subset W_1$
is invertible.

We will use the {\bf cobordism diagram} below.
To explain the cobordism diagram
we consider the shift map that brings $W_{\!\st}$ to $W_0$
in the construction we gave
in Section \ref{subsec:acobordism}.
The image of the section $s$,
which is contained in $\partial W_{\!\st}$,
is mapped to $S_0\subset M_0$ under the shift map.
Also, with the construction of $X$
we find a diffeomorphism $M_0\ra M_1$
that maps $S_0$ onto the section $S_1$, say,
such that $S_0$ is isotopic to $S_1$ inside $W_1$.
This isotopy extends to an isotopy of $W_0$ inside $W_1$
resulting in the following homotopy commutative diagram:

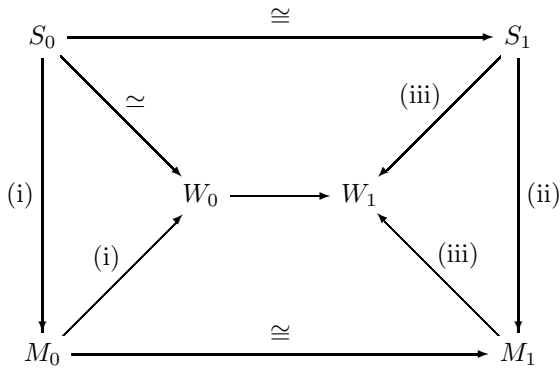
\begin{diagram}
S_0&&&\rTo^{\cong}&&& S_1\\
&\rdTo^{\simeq}&&&&\ldTo^{\mathrm{(iii)}}&\\
\dTo^{\mathrm{(i)}}&&W_0&\rTo&W_1&&\dTo_{\mathrm{(ii)}}\\
&\ruTo^{\mathrm{(i)}}&&&&\luTo^{\mathrm{(iii)}}&\\
M_0&&&\rTo^{\cong}&&&M_1
\end{diagram}
\vspace{0.4cm}

\noindent
We would like to understand the diagram in homology of degree $2d+1$.
Starting with the left triangle
we see that the inclusion of the $(2d+1)$-sphere $S_0$ into $W_0$
is a homotopy equivalence.
Using the Gysin sequence of the Euler class zero
bundle $ST^*S^{2d+1}\ra S^{2d+1}$ and the universal coefficient theorem
we see that the homology of $M_0$ in degree $2d+1$ equals $\Z$. 
Therefore,
all morphisms of the left triangle in degree $2d+1$
are multiplication by $\pm1$,
i.e.\ the maps (i) are $H_{2d+1}$-isomorphisms.

Using the outer square we see that
the map (ii) is a $H_{2d+1}$-isomorphism too.
We claim that the maps (iii) are $H_{2d+1}$-isomorphisms.
Our assumptions on the homology type of $W$
translate into $H_*W_1=H_*W_0$.
Poincar\'e--Lefschetz duality and the universal coefficient theorem imply
$H_{2d+2}(W_1, M_1)=0$
and $H_{2d+1}(W_1, M_1)=\Z$.
Therefore,
the long exact sequence for the pair $(W_1, M_1)$
implies exactness of 
\[
0
\ra H_{2d+1}M_1
\ra H_{2d+1}W_1
\ra\Z
\ra\Z
\ra 0
\,,
\]
where we used that $H_{2d}M_1=\Z$
by the Gysin sequence of the Euler class zero
bundle $ST^*S^{2d+1}\ra S^{2d+1}$ and the universal coefficient theorem.
As any surjection $\Z\ra\Z$ is necessarily invertible,
$H_{2d+1}M_1\ra H_{2d+1}W_1$
is indeed an isomorphism.
Therefore,
the maps labeled (iii) are isomorphisms in $H_{2d+1}$
as claimed.

Therefore,
all morphisms in the cobordism diagram
are $H_{2d+1}$-isomorphisms.
In particular, 
$\iota_*\co H_{2d+1}W_0\ra H_{2d+1}W_1$
is invertible.
This completes the proof of Lemma \ref{lem:hcobor}.

%%%%%%%%%%%%%%%%%%%%%%%%%%%%%%%%%%%%%%%%
%%%%%%%%%%%%%%%%%%%%%%%%%%%%%%%%%%%%%%%%

\section{Holomorphic filling via cap constructions \label{sec:holfillviacap}}

Cap constructions were used in \cite{geizeh12,sz17}
to verify instances of the Weinstein conjecture,
and in \cite{bgz} for classifications of subcritical fillings.
In the critical case too, symplectic caps are essential.

%%%%%%%%%%%%%%%%%%%%%%%%%%%%%%%%%%%%%%%%

\subsection{The anti-Hopf map and capping\label{sec:antihopf}}

The unit sphere
$S^{2d+1}\subset\C^{d+1}$, $d \geq 1$,
fibres over $\C P^d$ via the Hopf map
$S^{2d+1}\ra\C P^d$, $z\mapsto [z]$.
As observed by Audin--Lalonde--Polterovich (\cite{alp94})
the graph of the anti-Hopf map 
\[
S^{2d+1}\lra\C^{d+1}\times\C P^d\,,
\quad z\longmapsto (z, [\bar z])\,,
\]
defines a Lagrangian embedding,
which by the Weinstein neighbourhood theorem
extends to a symplectic embedding
of a neighbourhood of the zero section
of the cotangent bundle $T^*S^{2d+1}$. 
Rescaling the metric on $S^{2d+1}$ suitably,
this defines an embedding of $DT^*S^{2d+1}$
into $\C^{d+1} \times \C P^d$, and
we may assume that
\[
(M,\xi)=ST^*S^{2d+1}\subset\C^{d+1}\times\C P^d
\]
appears as a hypersurface of contact type.
The complement of the interior of $DT^*S^{2d+1}$
is a so-called {\bf symplectic cap}
\[
\cp:= \C^{d+1}\times\C P^d\setminus
\Int\big(DT^*S^{2d+1}\big)
\]
and carries the symplectic form
\[
\omega_{\cp}:=\rmd x\wedge\rmd y\oplus\omega_{\fs}\,.
\]
We define a symplectic manifold
\[
(Z,\Omega):=
(W,\omega)\cup_{(M, \xi)}
\big(\cp,\omega_{\cp}\big)
\]
via gluing along contact type boundaries.
Choose $z_0\in\C$ such that
the complex hypersurface
$z_0\times\C^d\times\C P^d$ is contained in $Z$.

\begin{lem}
\label{lem:surjectiveimpliesdiffeo}
If the embedding
$z_0\times\C^d\times\C P^d\ra Z$
is $\pi_1$-surjective and $H_k(\,.\,;\F)$-onto for all $k\in\Z$
and all fields $\F$,
then $W$ and $W_{\!\st}$ are diffeomorphic.
\end{lem}

\begin{proof}
In view of Proposition \ref{prop:homologyenough}
we need to show that $\pi_1W=1$
and $H_*W=H_*W_{\!\st}$.

By the $\pi_1$-surjectivity assumption,
we have $\pi_1Z=1$,
since $\C P^d$ is simply connected.
As $S^{2d+1}\subset\C^{d+1}\times\C P^d$
has codimension $\geq 3$,
we get that $\cp$
is simply connected.
By the Seifert--van Kampen theorem,
$\pi_1Z$ is the free product of $\pi_1W$ and the trivial group,
because $M$ is simply connected.
It follows that $\pi_1W=1$.

By assumption,
the induced map
$H_*(\C P^d;\F)\ra H_*(Z;\F)$
is surjective for all fields $\F$.
With the universal coefficient theorem,
we find $H_kZ=0$ for all $k\geq 2d+1$.
The Mayer--Vietoris sequence
for the decomposition $Z=W\cup\cp$ along $M$ gives
$H_kM=H_kW \oplus H_k(\cp)$
for all $k \geq 2d+1$.
Similarly,
$H_kM=H_kW_{\!\st} \oplus H_k(\cp)$
for all $k \geq 2d+1$,
so that $H_kW=H_kW_{\!\st}$
for all $k\geq 2d+1$.

To prove the isomorphism
$H_kW=H_kW_{\!\st}$
for all $k\leq2d$,
observe that by excision and Poincar\'e--Lefschetz duality
applied to $(Z,W)$ we have
$H^kW=H_{4d+2-k}(Z,\cp)$ for all $k\in\Z$.
Note that $k\leq2d$ translates into $4d+2-k \geq 2d+2$
and that the long exact sequence of the pair $(Z,\cp)$ gives
$H_m(Z, \cp)=H_{m-1}(\cp)$ for all $m\geq2d+2$.
Hence,
$H^kW=H_{4d+1-k}(\cp)$
for all $k\leq2d$.
Similarly,
$H^kW_{\!\st}=H_{4d+1-k}(\cp)$
for all $k\leq2d$
using excision of the pair
$(\C^{d+1}\times\C P^d,W_{\!\st})$
and Poincar\'e--Lefschetz duality.
We conclude $H^kW=H^kW_{\!\st}$
for all $k\leq2d$,
which in fact holds true in homology
with any field coefficients.
The universal coefficient theorem implies
$H_kW=H_kW_{\!\st}$ for all $k\leq2d$,
as the reduced homology of $S^{2d+1}$
vanishes in all degrees $\leq2d$.
\end{proof}

%%%%%%%%%%%%%%%%%%%%%%%%%%%%%%%%%%%%%%%%

\subsection{A cap construction\label{sec:acapcon}}

As a first step in defining a moduli space of holomorphic spheres
we will partially compactify $\C^{d+1}\times\C P^d$.
The subset
\[
B_r^2(0)\times \C^d \times \C P^d\,,
\quad r > 1\,,
\]
of $\C^{d+1}\times\C P^d$ completes symplectically to
\[
\Big(\C P^1\times\C^d \times\C P^d,
r^2 \omega_{\fs}\oplus\rmd x\wedge\rmd y\oplus\omega_{\fs}\Big)
\]
where $\omega_{\fs}$ stands for the Fubini--Study form
of the respective projective space.
The Lagrangian image of $S^{2d+1}$ under the anti-Hopf map
is contained in $B^2_{r}(0)\times\C^d\times\C P^d$.
By rescaling the round metric on $S^{2d+1}$
as in Section \ref{sec:antihopf},
$(M, \xi)=ST^*S^{2d+1}$ also appears as a hypersurface
of contact type in $\C P^1\times\C^d\times\C P^d$
provided with the symplectic form
\[
\hat{\omega}:=
r^2 \omega_{\fs}\oplus\rmd x\wedge\rmd y\oplus\omega_{\fs}\,.
\]
We equip the complement
\[
\widehat{\cp}:=
\C P^1\times\C^d\times\C P^d \setminus
\Int\big(DT^*S^{2d+1}\big)
\]
with the symplectic form
induced by $\hat{\omega}$
and obtain a symplectic cap.
•$\widehat{\cp}$ contains holomorphic spheres
\[
C_1:=\C P^1\times\pt\times\pt\,,
\qquad
C_2:=\pt\times\pt\times\C P^1\,,
\]
where $\pt$ stands for a generic choice
of a point in the respective factor.
In the definition of $C_2$,
$\C P^1\subset \C P^d$
denotes a complex line.
Dual hypersurfaces are
\[
H_1:=\infty\times\C^d\times\C P^d\,,
\qquad
H_2:=\C P^1\times\C^d\times\C P^{d-1}\,,
\]
where $\C P^{d-1}\subset\C P^d$ denotes a complex hyperplane.
Assuming $\C P^1$ and $\C P^{d-1}$
in general position in $\C P^d$
it follows for the intersection numbers that
\[
C_i\cdot H_j=\delta_{ij}
\]
for $i,j=1,2$.

%%%%%%%%%%%%%%%%%%%%%%%%%%%%%%%%%%%%%%%%

\subsection{A moduli space\label{sec:amodspace}}

We define a symplectic manifold
\[
\big(\hat{Z},\hat{\Omega}\big):=
(W,\omega)\cup_{(M, \xi)}\Big(\widehat{\cp},\hat{\omega}\Big)\,.
\]
It contains a subset $W_{\square}$
that by definition is the union of $W$
and all points in $\widehat{\cp}$ of the form
$(z_1,\bfz,w)$
that satisfy $|z_1|<R$ and $|\bfz|<R$
for some fixed $R\in(1,r)$.
Let $J$ be a tame almost complex structure on $(\hat{Z},\hat{\Omega})$
that is equal to the standard complex structure $J_{\st}$
on $\hat{Z} \setminus W_{\square}$
induced by $\C P^1\times\C^{d}\times\C P^d$.

\begin{lem}
\label{lem:maximumprinciple}
For all non-constant parametrised holomorphic spheres
$C=u(\C P^1)$, $u\co\C P^1\ra(\hat{Z},J)$,
the following hold true.
\begin{enumerate}
\item[(i)]
If $C\subset\hat{Z}\setminus W_{\square}$,
then $u=(u^1,\pt,u^3)$, where $\pt\in\C^{d}$
and $u^1\co\C P^1\ra\C P^1$, $u^3\co\C P^1\ra\C P^d$ are holomorphic.
\item[(ii)]
If $C\cap\big(\hat{Z}\setminus\overline W_{\square}\big)\neq\emptyset$
and $C\cap H_1=\emptyset$,
then $u=(\pt,\pt,u^3)$,
where $u^3: \C P^1 \ra \C P^d$ is holomorphic.
\end{enumerate}
\end{lem}

\begin{proof}
 (i) follows with the maximum principle.
 (ii) holds because the maximum principle
 in $\C$- and in $\C^d$-direction implies
 that the respective projections of $C$ {\it a posteriori}
 are contained in the level sets of the radial distance functions.
 Hence, $C\subset\hat{Z}\setminus W_{\square}$ and
 the claim follows with (i).
\end{proof}

The {\bf moduli space} $\MM$
consists of all holomorphic maps
$u\co\C P^1\ra(\hat{Z},J)$
such that $u(\C P^1)$ is homologous to $C_1$
and $u(z)\in z\times\C^d\times\C P^d$
for $z\in\{\pm\varrho,\infty\}$, $\varrho\in(R,r)$.

\begin{rem}
\label{rem:firstprop}
We observe:
\begin{enumerate}
\item [(i)]
  All holomorphic spheres $u\in\MM$ are simple
  because positivity of intersections implies $u\bullet H_1=C_1\cdot H_1=1$,
  see \cite{geizeh12,mcsa04}.
\item [(ii)]
  The index
  $2\dim_{\C}(\hat{Z})+2c_1(C_1)-6$
  of the underlying Fredholm problem equals $4d$,
  because
  $T\hat{Z}|_{C_1}=T\C P^1\oplus\underline{\C}^d\oplus\underline{\C}^d$
  and, hence, $c_1(C_1)=2$.
\item[(iii)]
  If $u(\C P^1)\subset\hat{Z}\setminus W_{\square}$
  for $u\in\MM$, then $u=(\id,\pt,\pt)$.
  In particular $u$ is regular.
  Indeed, 
  with Lemma \ref{lem:maximumprinciple} (i)
  we get $u=(u^1,\pt,u^3)$ for holomorphic spheres $u^1$ and $u^3$
  in $\C P^1$ and $\C P^d$, resp.
  Because of $u\bullet H_1=1$
  the holomorphic map $u^1$
  is an automorphism of $\C P^1$
  that fixes $3$ points,
  i.e.\ $u^1=\id$.
  Therefore,
  \[
  r^2\pi=
  \int_{\C P^1}u^*\hat{\Omega}=
  r^2\pi+\int_{\C P^1}(u^3)^*\omega_{\fs}
  \,.
  \]
  The second integral is the Dirichlet energy of $u^3$,
  which has to vanish by the equation.
  It follows that $u^3$ is constant.
\item[(iv)]
  We choose the almost complex structure $J$ on $\hat{Z}$
  such that $u\in\MM$ is regular
  whenever $u(\C P^1)\cap W_{\square}\neq\emptyset$.
  This can be done for example by perturbing $J$ on $W_{\square}$ suitably,
  see \cite{geizeh12,mcsa04}.
  With these generic choices
  the moduli space $\MM$ is a smooth orientable manifold
  of dimension $4d$.
\item[(v)]
  By the maximum principle in $\C^d$-direction 
  as used in Lemma \ref{lem:maximumprinciple} (ii)
  and item (iii) above,
  the end of $\MM$ is naturally diffeomorphic to
  the model
  $\big(\C^d\setminus B^{2d}_R(0)\big)\times\C P^d$
  via the correspondence between
  \[
   \Big\{u\in\MM\;\Big|\;u(\C P^1)\cap
   \big\{(z_1,\bfz,w)\in\widehat{\cp}\;\big|\;|\bfz|\geq R\big\}
   \neq\emptyset\Big\}
  \]
  and
  \[
   \big\{u=(\id,\bfz,w)\;\big|\;\bfz \in \C^d,\; |\bfz|\geq R,\;w\in\C P^d\big\}\,.
  \]
  We choose the orientation on $\MM$
  so that this correspondence preserves the orientations.
\end{enumerate}
\end{rem}

\begin{lem}
\label{lem:properthandiffeo}
If the evaluation map
$\widehat{\ev}\co\MM\times\C P^1\ra\hat{Z}$, $(u, z)\mapsto u(z)$
is proper,
then $W$ and $W_{\!\st}$ are diffeomorphic.
\end{lem}

\begin{proof}
By definition of $\MM$,
each holomorphic sphere $u(\C P^1)$, $u\in\MM$,
intersects $H_1$ at $\infty\in\C P^1$.
This intersection is unique by positivity of intersections.
Therefore,
the evaluation map $\widehat{\ev}$ restricts to 
an evaluation map on $\MM\times\C$
that takes values in $Z=\hat{Z}\setminus H_1$,
i.e.\ $(u, z)\mapsto u(z)$ is a well-defined map
$\ev\co\MM\times\C\ra Z$.

As $\widehat{\ev}$ is proper,
so is the restricted map $\ev$,
hence $\ev$ is of mapping degree $1$,
see Remark \ref{rem:firstprop} (v).
Using the Umkehr-homomorphism based on Poincar\'e duality
involving compactly supported cohomology,
we see that $\ev$ is $H_*(\,.\,,\F)$-surjective for all fields $\F$.
Surjectivity of $\ev$ in $\pi_1$ follows with a covering
that corresponds to the index of the $\ev_{\#}$-image in $\pi_1(Z)$. 

In view of Lemma \ref{lem:surjectiveimpliesdiffeo}
we consider the following commutative diagram:
\begin{diagram}
\MM\times\C&&\rTo^{\ev}&&Z\\
\uTo^{\simeq}&&&&\uTo^{\mathrm{\iota}}\\
\MM\times\varrho&&\rTo^{\ev}&&\varrho\times\C^d\times\C P^d
\end{diagram}
\vspace{0.2cm}

\noindent
The induced diagram in $\pi_1$ and $H_*(\,.\,,\F)$
is commutative with upper $\ev_{\#}$ and $\ev_*$ surjective
for all fields $\F$.
Hence,
$\iota_{\#}$ and $\iota_*$ are surjective.
\end{proof}

%%%%%%%%%%%%%%%%%%%%%%%%%%%%%%%%%%%%%%%%

\subsection{Properness of the evaluation map\label{sec:properev}}

We will investigate properness of the evaluation map.
The results of this section are part of the proof
of Theorem \ref{thm:installabilitytheorem}.
For basics about holomorphic spheres we refer to
McDuff--Salamon's book \cite{mcsa04}.

\subsubsection{Gromov convergence}
By the description of the end of $\MM$
in Remark \ref{rem:firstprop} (v)
it suffices to consider sequences $\{u^{\nu}\}_{\nu}$
of holomorphic spheres $u^{\nu}\in\MM$
whose images $u^{\nu}(\C P^1)$
are contained in
\[
W\cup
\Big\{(z_1,\bfz,w)\in\widehat{\cp}\;\big|\;|\bfz|\leq R\Big\}\,.
\]
Because the energy
\[
E(u^{\nu})
=\int_{\C P^1}(u^{\nu})^*\hat{\Omega}
=\pi r^2 
\]
is uniformly bounded for all $\nu\in\N$,
Gromov's compactness theorem applies:
$u^{\nu}$ converges to a stable holomorphic sphere
$\bfu$ with three marked points fixed.
We denote by $u_1,\dots,u_N$, $N\geq1$,
the non-constant components of $\bfu$.
The convergence is in $C^{\infty}$ precisely if $N=1$.
Therefore,
the evaluation map $\widehat{\ev}$ is proper
if $N=1$.

\subsubsection{Intersections}
\label{intersectionpattern}
The intersection number of $u^{\nu}$ and $H_1$
is $u^{\nu}\bullet H_1=1$ for all $\nu\in\N$.
Furthermore a neighbourhood of $H_1$
is foliated by complex hypersurfaces
of the form $z_1\times\C^d\times\C P^d$
with $|z_1|$ sufficiently large.
Hence, we can assume that
none of the $u_j(\C P^1)$, $j=1,\dots,N$,
lies entirely in $H_1$,
so that the intersection numbers
of the bubbles $u_j$ and $H_1$
are defined, non-negative by positivity of intersections,
i.e.\ $u_j\bullet H_1\geq 0$ for all $j=1,\ldots,N$,
and sum up to
$\sum_{j=1}^N u_j\bullet H_1=1$
according to Gromov convergence.
Relabelling indices, if necessary, yields
\[
u_1\bullet H_1=1
\qquad
\text{and}
\qquad
u_j\bullet H_1=0
\]
for all $j\geq 2$.
By positivity of intersections
and the maximum principle
this implies that for each $j\geq 2$
we have the following alternative:
either
$u_j(\C P^1)\subset W_{\square}$
or $u_j=(\pt,\pt,u^3_j)$ for some
holomorphic $u^3_j: \C P^1 \ra \C P^d$,
see Lemma \ref{lem:maximumprinciple} (ii).
In particular,
\[
u_j(\C P^1)\subset Z_R\,,
\]
for all $j\geq 2$,
where
\[
Z_R:=W\cup
\Big\{(z_1,\bfz,w)\in\cp\big|\;|\bfz|\leq R\Big\}\,.
\]
This will be used in the following section.

\subsubsection{Implementing asphericity}
\label{posofenergy}
We consider the homology classes
$[u_1]\in H_2\hat{Z}$
and
$[u_j]\in H_2Z_R$, $j\geq 2$,
of the bubble spheres.
Because the contact hypersurface $M=ST^*S^{2d+1}$
is simply connected, we can decompose
\[
[u_1]=
S_1+A_1
\in H_2W\oplus H_2\widehat{\cp}
\]
and
\[
[u_j]=
S_j+A_j
\in H_2W\oplus H_2\big(Z_R\setminus\Int W\big)
\,,
\]
for all $j\geq 2$,
using a general position argument:
The bubble spheres intersect $M$ along closed loops.
These loops appear as the boundary of disc maps in $M$
for both orientations of the loops.
Adding the disc maps
to the disc parts of the bubble spheres corresponding to $W$
and its complements -- respecting orientations -- 
yields sphere maps that represent the
$S_j$ and $A_j$ as desired.
Since $(W,\omega)$ is symplectically aspherical,
$\hat{\Omega}(S_j)=0$ for all $j=1,\dots,N$.
Consequently,
using Stokes' theorem, we have
\[
0<E(u_j)=\hat{\Omega}(A_j)
\]
for all $j=1,\dots,N$.

\subsubsection{Enumeration}
\label{subsubsec:enum}
In order to compute
the symplectic energy of the holomorphic spheres $u_j$
it suffices to compute $\hat{\Omega}(A_j)$
for the homology $2$-classes $A_j$ in $\widehat{\cp}$,
see \ref{posofenergy}.
Invoking the intersection pattern
from \ref{intersectionpattern}
we find integers $\ell_1,\ldots,\ell_N$
such that
\[
A_1=[C_1]+\ell_1[C_2]
\qquad
\text{and}
\qquad
A_j=\ell_j[C_2]\,,
\]
for all $j\geq 2$.
Indeed, reading this as an equation in $H_2\widehat{\cp}$
this equation holds literally true if $d\geq2$;
if $d=1$, the equation holds up to
possibly adding a multiple of the fibre class
in $M=ST^*S^3\cong S^3\times S^2$.
However,
any representing fibre $S^2$
is disjoint from the complex hypersurface $H_1$
and does not contribute to the symplectic energy
by local exactness of the symplectic form near $M$.
For that reason we will not spell out the fibre class summands.

Positivity of energy from \ref{posofenergy}
yields $r^2+\ell_1 > 0$ and $\ell_j >0$ for all $j\geq 2$.
Furthermore,
using $\hat{\Omega}(A_j)=\hat{\omega}(A_j)$ for all $j$,
\[
\pi r^2=
E(u^{\nu})=
\sum_{j=1}^N\hat{\Omega}(A_j)=
\pi r^2+\ell_1\pi+\ell_2\pi+\cdots+\ell_N\pi
\,.
\]
It follows that
$\ell_1+\cdots+\ell_N=0$.
Therefore,
\[
N-1\leq
\ell_2+\cdots+\ell_N=
-\ell_1< 
r^2
\,.
\]
Taking $r\in (1,\sqrt{2}]$ in Section \ref{sec:acapcon},
the cases $N=1$ and $N=2$ remain.
If $N=1$, then the evaluation map $\widehat{\ev}$
will be proper.
If $N=2$, the above inequality turns into
$1=N-1\leq\ell_2=-\ell_1<2$
so that $\ell_2=-\ell_1=1$,
i.e.\
\[
A_1=[C_1]-[C_2]
\qquad
\text{and}
\qquad
A_2=[C_2]\,.
\]
The symplectic energies of
the two bubble spheres satisfy
\[
E(u_1)=(r^2 -1)\pi
\qquad
\text{and}
\qquad
E(u_2)=\pi\,.
\]

%%%%%%%%%%%%%%%%%%%%%%%%%%%%%%%%%%%%%%%%
%%%%%%%%%%%%%%%%%%%%%%%%%%%%%%%%%%%%%%%%

\section{Fittings\label{sec:installability}}

Recall the cap construction from Section \ref{sec:antihopf},
that
\[
(Z,\Omega)=
(W,\omega)\cup_{(M, \xi)}
\big(\cp,\omega_{\cp}\big)
\,,
\]
and consider the complex hypersurface
$H:=\C^{d+1}\times\C P^{d-1}$
in $\C^{d+1}\times\C P^d$.

%%%%%%%%%%%%%%%%%%%%%%%%%%%%%%%%%%%%%%%%

\subsection{Definition\label{sec:definition}}

A symplectic filling $(W,\omega)$
of $(M,\xi)=ST^*S^{2d+1}$ is called a {\bf fitting}
if there exist a closed neighbourhood $U$ of $H$
in $\C^{d+1}\times\C P^d$ and
an embedding $\varphi\co U\ra Z$
with $\varphi|_{U\cap\cp}=\id$
such that
$\varphi_*J_{\st}$ is tamed by $\Omega$ on $\varphi(U)$
and the homological intersection
\[
\varphi_*[H\cap DT^*S^{2d+1}]\cdot S=0
\]
for all spherical classes $S\in H_2W$,
i.e.\ for all classes in the image of the Hurewicz homomorphism.

%%%%%%%%%%%%%%%%%%%%%%%%%%%%%%%%%%%%%%%%

\subsection{Proof of the Theorem
\ref{thm:installabilitytheorem}\label{sec:pfofthm}}

We show that $W$ and $W_{\!\st}$ are diffeomorphic
provided that $(W,\omega)$
is a symplectically aspherical fitting of $(M,\xi)$.
We consider $(\hat{Z},\hat{\Omega},J)$
constructed in Section \ref{sec:amodspace} 
and choose $r\in(1,\sqrt2)$,
see \ref{subsubsec:enum}.
The embedding $\varphi$ extends to an embedding
(again denoted by $\varphi$)
of a neighbourhood of $H_2=\C P^1\times\C^d\times\C P^{d-1}$ in
$\C P^1\times\C^d\times\C P^d$ into $\hat{Z}$.

This extended embedding can be assumed to be holomorphic:
The almost complex structure $J$ on $\hat{Z}$
equals the standard complex structure $J_{\st}$ on
$\hat{Z}\setminus W_{\square}$. 
We require $J$ to be $\varphi_*J_{\st}$ in $\varphi(U)$ and
perturb $J$ on $W_{\square}\setminus\varphi(U)$
to be generic for all holomorphic spheres
that intersect the open set $W_{\square}\setminus\varphi(U)$ non-trivially.

In order to guarantee transversality of the moduli space $\MM$
we need to verify regularity for all $u\in\MM$
with $u(\C P^1)\subset\varphi(U)$.
For those $u$, we obtain a holomorphic map
\[
\varphi^{-1}\circ u\co\C P^1\lra\C P^1\times\C^d\times\C P^d
\,,
\]
which by the maximum principle can be written as
\[
\varphi^{-1}\circ u=(k,\pt,h)
\]
for holomorphic sphere maps $k$ into $\C P^1$
and $h$ into $\C P^d$.
It suffices to show
that $k$ is an automorphism of $\C P^1$
(and hence $k=\id$ by the definition of $\MM$)
and that $h$ is constant:
We compute symplectic energies
of $\varphi^{-1}\circ u$
with respect to the possibly different symplectic forms
$\varphi^*\hat{\Omega}$ and $\hat{\omega}$.
The symplectic energy of $u$ equals
\[
\pi r^2=\hat{\Omega}([u])=\hat{\Omega}(A)
\]
by asphericity
of $\big(\varphi(U)\cap W,\omega\big)$
and a splitting
$[u]=S+ A\in H_2W\oplus H_2\widehat{\cp}$,
see \ref{posofenergy}.
Because $\hat{\Omega}$ and $\hat{\omega}$
coincide on $\widehat{\cp}$, this reads as
$\pi r^2=\hat{\omega}(A)$.
On the other hand,
the symplectic energy of
$\varphi^{-1}\circ u$ with respect to $\hat{\omega}$
equals
\[
\hat{\omega}\big([\varphi^{-1}\circ u]\big)=
\hat{\omega}(A)
\,,
\]
because $[\varphi^{-1}\circ u]=(\varphi^{-1})_*S+ A$
in $H_2W_{\!\st}\oplus H_2\widehat{\cp}$.
This is because $\varphi|_{U\cap\widehat{\cp}}=\id$
and $W_{\!\st}=DT^*S^{2d+1}$ has vanishing
second homology.
We see that
\[
\pi r^2=
\hat{\omega}(A)=
r^2\int_{\C P^1}k^*\omega_{\fs}+\int_{\C P^1}h^*\omega_{\fs}=
a\pi r^2+b\pi
\]
for some non-negative integers $a,b$.
It follows that
\[
r^2(1-a)=b
\,.
\]
Since $1<r^2<2$ is not an integer
$1-a$ must vanish,
so that $a=1$ and $b=0$.
This implies that the energy of $k$ is $\pi r^2$,
i.e.\ $k$ is an automorphism on $(\C P^1,r^2\omega_{\fs})$,
and that $h$ has zero energy and, hence, must be constant.
Therefore,
$u$ is of the form $u=(\id,\pt,\pt)$ and,
in particular, $u$ is regular. 

To finish the proof of Theorem \ref{thm:installabilitytheorem}
it remains to show properness of the evaluation map $\widehat{\ev}$,
see Lemma \ref{lem:properthandiffeo}.
In view of Section \ref{sec:properev}
we consider the case $N=2$
and $u_1$ as described.
Suppose that $u_1(\C P^1)$
is not contained in the complex hypersurface
$\varphi(H_2)\subset\hat{Z}$.
By positivity of intersections, the intersection number
\[
\varphi(H_2)\bullet u_1\geq 0
\]
is defined and non-negative.
On the other hand, $\varphi(H_2)\bullet u_1$
is equal to the homological intersection
of $[\varphi(H_2)]=\varphi_*[H_2]$ with $[u_1]$.
As in \ref{subsubsec:enum}
we use $A_1=[C_1]-[C_2]$ for the second summand
of the decomposition $[u_1]=S_1+A_1$.
If $d\geq2$ such an equation is literally true;
the second homology of $\widehat{\cp}$
is the one of $\C P^1\times\C^d\times\C P^d$
by general position.
Hence,
$[u_1]=S_1+[C_1]-[C_2]$.
The homological intersection condition
in the definition of a fitting yields
$\varphi_*[H_2]\cdot S_1=0$ as $S_1$ is spherical;
the duality relations in Section \ref{sec:acapcon} yield
\[
\varphi_*[H_2]\cdot [C_1]=0\,,
\qquad
\varphi_*[H_2]\cdot [C_2]=1\,.
\]
It follows that
\[
[\varphi(H_2)]\cdot [u_1]=-1
\]
contradicting positivity of intersections.
If $d=1$, the equation $[u_1]=S_1+[C_1]-[C_2]$
holds up to adding a multiple of the fibre class
in $ST^*S^3\cong S^3\times S^2$.
Again by the definition of a fitting
the fibre class contributes zero
to the intersection with $[\varphi(H_2)]$.
Therefore, we reach the same contradiction.

Consequently,
$u_1(\C P^1)\subset\varphi(H_2)$.
Analysing the energy $E(u_1)=(r^2-1)\pi$
in terms of $\varphi^{-1}\circ u_1=(k,\pt,h)$ as before
we reach the equality
\[
(r^2-1)\pi=a'\pi r^2+ b'\pi
\]
for non-negative integers $a',b'$.
In other words,
\[
r^2(1-a')=1+b'
\]
with $1<r^2<2$.
This is impossible.
Therefore, $N=1$ and $\widehat{\ev}$ is proper.
By Lemma \ref{lem:properthandiffeo}
$W$ is diffeomorphic to $W_{\!\st}$.
\hfill
Q.E.D.
\bigskip

%%%%%%%%%%%%%%%%%%%%%%%%%%%%%%%%%%%%%%%%%%%%%%%
%%%%%%%%%%%%%%%%%%%%%%%

\begin{ack}
  We thank Hansj\"org Geiges for suggesting to use the term {\it fitting},
  and his comments on the first draft of the manuscript.
\end{ack}

%%%%%%%%%%%%%%%%%%%%%%%%%%%%%%%%%%%%%%%%

\end{document}